\documentclass[a4paper,12pt]{amsart}

\pdfoutput=1
\usepackage[hidelinks,pdfstartview=FitH]{hyperref} 
\usepackage{amssymb,tikz}
\usepackage[all,cmtip]{xy}
\usepackage{verbatim}
\usepackage{fullpage}
\usepackage{soul}

\allowdisplaybreaks[4]

\newtheorem{proposition}{Proposition}[section]
\newtheorem{lemma}[proposition]{Lemma}
\newtheorem{remark}[proposition]{Remark}
\newtheorem{theorem}[proposition]{Theorem}

\newtheorem{definition}[proposition]{Definition}

\newtheorem*{theorem*}{Theorem}
\newtheorem*{corollary*}{Corollary}

\numberwithin{equation}{section}

\renewcommand{\emptyset}{\varnothing}

\title[Pullbacks of graph C*-algebras]{\vspace*{-15mm}Pullbacks of graph C*-algebras\\ 
from admissible pushouts of graphs}

\begin{document}
\baselineskip14pt
\parskip=0.75\baselineskip
\parindent=7.5mm

\keywords{graph algebra, pushout and pullback, gauge action, quantum spaces: spheres, balls, lens spaces, weighted projective spaces}
\subjclass[2010]{46L45, 46L55,  46L85.}

\maketitle

\begin{center}
\MakeUppercase{Piotr M.~Hajac}\\
{\em Instytut Matematyczny, Polska Akademia Nauk\\
\'Sniadeckich 8, 00-656 Warszawa, Poland\\
E-mail: pmh@impan.pl\\
and\\
Department of Mathematics\\
University of Colorado Boulder\\
2300 Colorado Avenue
Boulder, CO 80309-0395
USA}
\end{center}

\begin{center}
\MakeUppercase{Sarah Reznikoff} \\
{\em Mathematics Department, Kansas State University\\
138 Cardwell Hall, Manhattan, KS 66502 US\\
E-mail: sarahrez@ksu.edu}
\end{center}

\begin{center}
\MakeUppercase{Mariusz Tobolski}\\
{\em Instytut Matematyczny, Polska Akademia Nauk\\
\'Sniadeckich 8, 00-656 Warszawa, Poland\\
E-mail: mtobolski@impan.pl}
\end{center}

\begin{abstract}
We define an admissible decomposition of a graph $E$ into subgraphs $F_1$ and~$F_2$, and consider the intersection graph 
$F_1\cap F_2$
as a subgraph of both $F_1$ and~$F_2$. We prove that, if the graph $E$ is row finite and its decomposition into the subgraphs 
$F_1$ and  $F_2$ is admissible, then the graph C*-algebra $C^*(E)$ of $E$ is the pullback C*-algebra  
of the canonical surjections from $C^*(F_1)$ and $C^*(F_2)$ onto $C^*(F_1\cap F_2)$.
\end{abstract}
\bigskip

\section{Introduction and preliminaries}

Pushouts of graphs have proven to be very useful in the theory of free groups~\cite{jr-s83}. 
We hope that our approach to pullbacks of graph algebras through
pushouts of underlying graphs will also turn out to be beneficial.

A graph C*-algebra is the universal C*-algebra associated to a directed graph.
If one considers a specific class of morphisms of directed graphs (e.g., see~\cite[Definition~1.6.2]{aasm17}), 
then the graph C*-algebra construction yields a covariant functor from the 
category of directed graphs to the category of C*-algebras.
On the other hand, Hong and Szyma\'nski \cite{hs08} showed that a pushout diagram in the category of 
directed graphs can lead to a pullback of C*-algebras. The purpose of this paper is to find 
conditions on the pushout diagram of graphs that give rise to the pullback diagram of the
associated graph \mbox{C*-algebras}. This leads to a notion of an admissible decomposition of a directed graph,
which we present in Section~2. The main result is contained in Section~3 and examples are in Section~4.

Our result is closely related to \cite[Corollary~3.4]{kpsw16}, where it is proven, in an appropriate form, 
for $k$-graphs without sinks but not necessarily row finite.
Herein, we focus our attention on row-finite 1-graphs but possibly with sinks. 
Thus our results are complementary and lead to the following question:
Is it possible to get rid of both of these assumtions (``row finite'' and ``no sinks'') 
at the same time to prove a more general pushout-to-pullback theorem?

In this paper, by a graph $E$ we will always mean a {\em directed graph}, i.e.~a~quad\-ruple 
$(E^0,E^1,s_E,r_E)$, where $E^0$ is the set of vertices,
$E^1$ is the set of edges, $s_E:E^1\to E^0$ is the source map and $r_E:E^1\to E^0$ is the range map.
A graph $E$ is called {\em row finite} if each vertex emits only a finite  number of edges. 
Next, $E$~is called~{\em finite} if both $E^0$ and $E^1$ are finite. A vertex is called a \emph{sink} if it does not emit any edge.
By a {\em path} $\mu$ in $E$ of length $|\mu|=k>0$ we mean a sequence of composable edges 
$\mu=e_1e_2\ldots e_k$. We treat vertices as paths of length zero. The set of all finite paths for a graph $E$ is denoted
by ${\rm Path}(E)$. One extends the source and the range maps to ${\rm Path}(E)$ in a natural way.
We denote the extended source and range maps by $s_{P\!E}$ and $r_{P\!E}$, respectively.

\begin{definition}
The \emph{graph \mbox{C*-algebra}} $C^*(E)$ of a row-finite graph $E$ is the universal \mbox{C*-algebra} generated by mutually 
orthogonal projections $P:=\big\{P_v\;|\;v\in E^0\big\}$ and partial isometries $S:=\big\{S_e\;|\;e\in E^1\big\}$
 satisfying the \emph{Cuntz--Krieger relations}~\cite{ck80}:
\begin{align}
S_e^*S_e &=P_{r_E(e)} && \text{for all }e\in E_1\text{, and}\tag{\text{CK1}} \label{eq:CK1} \\
\sum_{e\in s^{-1}_E(v)}\!\! S_eS_e^*&=P_v && \text{for all }v\in E_0\text{ that are not sinks.}\tag{CK2} \label{eq:CK2}
\end{align}
The datum $\{S,P\}$ is called a \emph{Cuntz--Krieger $E$-family}.
\end{definition}
\noindent
One can show that the above relations imply the standard path-algebraic relations:
\begin{equation}\label{pathalg}
S_f^*S_e=0 \quad\text{for}\quad e\neq f,\quad P_{s_E(e)}S_e=S_e=S_eP_{r_E(e)}\,.
\end{equation}

Any graph C*-algebra $C^*(E)$ can be endowed with a natural circle action 
\begin{equation}
\alpha:U(1)\longrightarrow\mathrm{Aut}(C^*(E))
\end{equation}
defined by its values on the generators:
\begin{equation}
\alpha_\lambda(P_v)=P_v\;,\qquad
\alpha_\lambda(S_e)=\lambda S_e\;,\quad\text{where}\quad \lambda\in U(1),\quad v\in E_0,\quad e\in E_1.
\end{equation}
The thus defined circle action is called the \emph{gauge action}.

A subset $H$ of $E^0$ is called {\em hereditary} iff, for any $v\in H$ such that there 
is a path starting at $v$ and ending at $w\in E^0$, we have $w\in H$. (Note that we can equivalently define the property of being hereditary
by replacing ``path'' with ``edge''.)
A subset $H$ of $E^0$ is called {\em saturated} iff there does \emph{not}
exist a vertex $v\notin H$ such that
\begin{equation}
0<|s_E^{-1}(v)|<\infty\qquad\text{and}\qquad r_E(s_E^{-1}(v))\subseteq H.
\end{equation}

Saturated hereditary subsets play a fundamental role in the theory of gauge-invariant 
ideals of graph C*-algebras.
It follows from \cite[Lemma~4.3]{bprs00} that, for any hereditary subset~$H$,
the algebraic ideal 
generated by $\{P_v\;|\;v\in H\}$ is of the form
\begin{equation}\label{herideal}
I_E(H)=\mathrm{span}\left\{S_xS_y^*\;|\;x,y\in\text{Path}(E)\,,\;r_{P\!E}(x)=r_{P\!E}(y)\in H\right\}.
\end{equation}
  \noindent
Here, for any path $\mu=e_1\ldots e_k$, we adopt the notation $S_\mu:=S_{e_1}\ldots S_{e_k}$.  Furthermore, if $\mu$ is a vertex,
then  $S_\mu:=P_\mu$.

By \cite[Theorem~4.1~(b)]{bprs00}, quotients by closed ideals generated by saturated hereditary subsets can also
be realised  as graph C*-algebras 
by constructing a quotient graph. Given a hereditary subset $H$ of $E^0$, the \emph{quotient graph} $E/H$ is 
given by
\begin{equation}
(E/H)^0:=E^0\setminus H\qquad\text{and}\qquad(E/H)^1:=E^1\setminus r_E^{-1}(H).
\end{equation}
Note that the restriction-corestriction of the range map $r_E$
to $(E/H)^1\to (E/H)^0$ makes sense for any~$H$, but the same restriction-corestriction of the source map $s_E$ exists  because 
$H$ is hereditary.

Moreover, if $H$ is also saturated, we obtain the
*-isomorphism
\begin{equation}\label{quotient}
C^*(E)\slash \overline{I_{E}(H)}\cong C^*(E/H),
\end{equation}
where $ \overline{I_{E}(H)}$ is the norm closure of $I_E(H)$. 

\section{Admissible decompositions of graphs}

Given two graphs $E=(E^0,E^1,s_E,r_E)$ and $G=(G^0,G^1,s_G,r_G)$, 
one can define a {\em graph morphism} $f:E\to G$ as a pair
of mappings $f^0:E^0\to G^0$ and $f^1:E^1\to G^1$ satisfying
\begin{equation}
s_G\circ f^1=f^0\circ s_E\quad\text{and}\quad r_G\circ f^1=f^0\circ r_E\,.
\end{equation}
We call the thus obtained category the
 {\em category of directed graphs}. 
 
A {\em subgraph} of a graph $E=(E^0,E^1,s_E,r_E)$ is a graph $F=(F^0,F^1,s_F,r_F)$
such that 
\begin{equation}
F^0\subseteq E^0, \quad F^1\subseteq E^1,\quad \forall\;e\in F^1\colon s_F(e)=s_E(e)\;\text{and}\; r_F(e)=r_E(e).
\end{equation}
Next, let $F_1$ and $F_2$ be two subgraphs of a graph~$E$. We define their {\em intersection} and \emph{union} as follows:
\begin{gather}
F_1\cap F_2:=(F_1^0\cap F_2^0,F_1^1\cap F_2^1,s_\cap,r_\cap),\nonumber\\
\forall\; e\in F_1^1\cap F_2^1:\; s_\cap(e):=s_E(e),~r_\cap(e):=r_E(e),\nonumber\\
F_1\cup F_2:=(F_1^0\cup F_2^0,F_1^1\cup F_2^1,s_\cup,r_\cup),\nonumber\\
\forall\; e\in F_1^1\cup  F_2^1 :\; s_\cup(e):=s_E(e),\; r_\cup(e):=r_E(e).
\end{gather}

To consider pushout diagrams in the category of directed graphs, we
follow the convention used in~\cite{ek79}.
 If a graph $E$ has two subgraphs $F_1$ and $F_2$
such that $E=F_1\cup F_2$, then the following diagram
\begin{equation}\label{}
\begin{gathered}
\xymatrix{
& E \\
F_1 \ar[ru]
& 
& F_2 \ar[lu]\\
& F_1\cap F_2 \ar[lu] \ar[ru]
}
\end{gathered}
\end{equation}
is automatically a pushout diagram.
Let us illustrate the concept of a pushout diagram of  graphs with the following example:
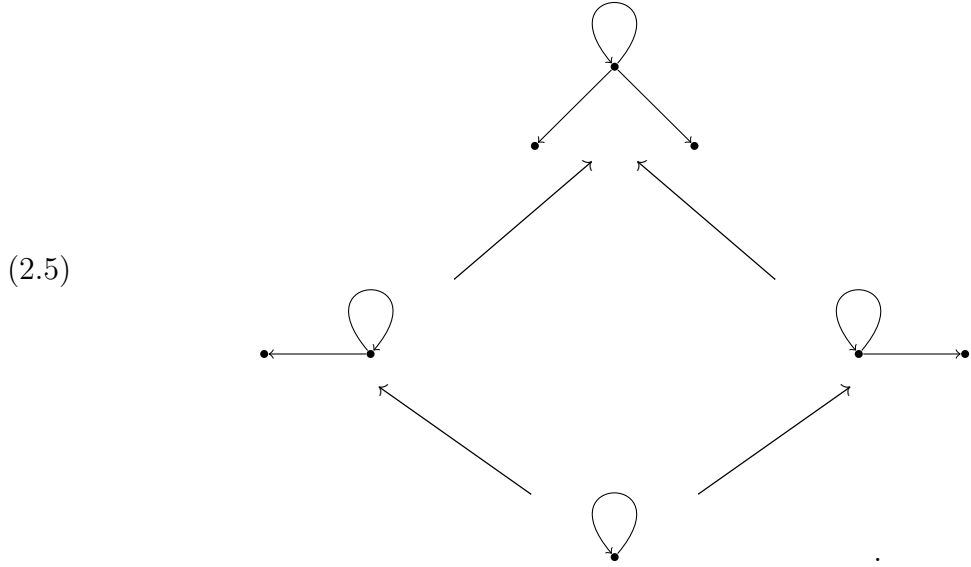
\begin{equation}\label{podles}
\begin{gathered}
\xymatrix{
&
\begin{tikzpicture}[scale=0.7,auto,swap]
\centering
\tikzstyle{vertex}=[circle,fill=black,minimum size=3pt,inner sep=0pt]
\tikzstyle{edge}=[draw,->]
\tikzset{every loop/.style={min distance=20mm,in=130,out=50,looseness=40}}
    \node[vertex] (1) at (0,0) {};
    \node[vertex] (2) at (-1.5,-1.5) {};
    \node[vertex] (3) at (1.5,-1.5) {};
    \path (1) edge [edge, loop above] node {} (1);
    \path (1) edge [edge] node {} (2);
    \path (1) edge [edge] node {} (3);
\end{tikzpicture}
&\\
\begin{tikzpicture}[scale=0.7,auto,swap]
\centering
\tikzstyle{vertex}=[circle,fill=black,minimum size=3pt,inner sep=0pt]
\tikzstyle{edge}=[draw,<-]
\tikzset{every loop/.style={min distance=20mm,in=130,out=50,looseness=40}}
    \node[vertex] (1) at (0,0) {};
    \node[vertex] (2) at (2,0) {};
    \path (1) edge [edge] node {} (2);
    \path (2) edge [edge, loop above] node {} (1);
\end{tikzpicture}
\ar[ur]& & 
\begin{tikzpicture}[scale=0.7,auto,swap]
\centering
\tikzstyle{vertex}=[circle,fill=black,minimum size=3pt,inner sep=0pt]
\tikzstyle{edge}=[draw,->]
\tikzset{every loop/.style={min distance=20mm,in=130,out=50,looseness=40}}
    \node[vertex] (1) at (0,0) {};
    \node[vertex] (2) at (2,0) {};
    \path (1) edge [edge, loop above] node {} (1);
    \path (1) edge [edge] node {} (2);
\end{tikzpicture}
\ar[ul]\\
&
\begin{tikzpicture}[scale=0.7,auto,swap]
\centering
\tikzstyle{vertex}=[circle,fill=black,minimum size=3pt,inner sep=0pt]
\tikzstyle{edge}=[draw,->]
\tikzset{every loop/.style={min distance=20mm,in=130,out=50,looseness=40}}
    \node[vertex] (1) at (0,0) {};
    \path (1) edge [edge, anchor=center, loop above] node {} (1);
\end{tikzpicture}
\ar[ur]\ar[ul]&
.}
\end{gathered}
\end{equation}

We are now  ready  to define an admissible decomposition of a row-finite graph:
\begin{definition}\label{decompose}
An unordered pair $\{F_1,F_2\}$ of subgraphs of a~row-finite graph $E$  is called an {\em admissible decomposition} of $E$
 iff the following conditions are satisfied:
\begin{enumerate}
\item $E=F_1\cup F_2$\,, \label{push1}
%\item if an edge $e\in F_i^1$ ends in $F_1^0\cap F^0_2$\,, then its source is also in $F_1^0\cap F^0_2$\,, $ i=1,2$,\label{her1}
%\item if $v$ is a sink in $F_i$, then $v$ is a sink in $E$, $ i=1,2$,\label{sat1}
\item if $v$ is a sink in $F_1\cap F_2$\,, then $v$ is a sink in $F_i$\,, $ i=1,2$,\label{sat2}
\item $F_1^1\cap F^1_2= r^{-1}_{F_i}(F_1^0\cap F^0_2)$, $i=1,2$.
\label{edgecon}
\end{enumerate}
\end{definition}
\noindent
Note that, by \eqref{push1} in Definition~\ref{decompose}, 
$E$ is a pushout of $F_1$ and $F_2$ over their intersection.
Observe also that Diagram~\eqref{podles} gives an example of an admissible decomposition of a graph.

Definition~\ref{decompose} prompts the following two lemmas.
\begin{lemma}\label{quot}
Let $\{F_1,F_2\}$ be an admissible decomposition of a row-finite graph $E$. Then 
$F_1\cap F_2=F_i/(F^0_i\setminus (F^0_1\cap F^0_2))$ and $F_i=E/(E^0\setminus F_i^0)$,
for $i=1,2$.
\end{lemma}
\begin{proof}
First, note that $F_i^0\setminus (F_1^0\cap F_2^0)$ is hereditary in $F_i$.
%by Definition~\ref{decompose}(3).
 Indeed, take $e\in F^1_i$. Then $s_{F_i}(e)\notin F_1^0\cap F_2^0$ implies $e \notin F_1^1 \cap F_2^1$.
 Hence, by 
 Definition~\ref{decompose}(3), we have
\mbox{$r_{F_i}(e)\!\notin\! F_1^0\!\cap\! F_2^0$}.
Therefore, we can define $F_i/(F_i^0\setminus (F^0_1\cap F^0_2))$,
which coincides with $F_1\cap F_2$ due to Definition~\ref{decompose}(3).

Next, note that 
\begin{equation}
E^0\setminus F_i^0=(F^0_i\cup F^0_j)\setminus  F_i^0= F_j^0\setminus  F_i^0=F_j^0\setminus  (F^0_i\cap F^0_j),
\end{equation} 
where $j\neq i$ and $j=1,2$, so we already know that $E^0\setminus F_i^0$ is hereditary in~$F_j$. To see that it is hereditary in $E$, we only
need to exclude edges starting in $E^0\setminus F_i^0$ and ending in $E^0\setminus F_j^0$. They do not exist because 
$E^1=F^1_i\cup F^1_j$,
so $E^0\setminus F_i^0$ is hereditary in~$E$.

It remains to verify that $F^1_i=r^{-1}_E(F_i^0)$. 
To this end, taking advantage of the admissibility of $(F_i\cap F_j)\subseteq F_i$, we compute
\begin{equation}
r_E^{-1}(F_i^0)\setminus F^1_i=r_{F_j}^{-1}(F_i^0)\setminus F^1_i
=r_{F_j}^{-1}(F_i^0\cap F_j^0)\setminus F^1_i=(F^1_i\cap F_j^1)\setminus F^1_i =\emptyset.
\end{equation}
Therefore, as $F_i^1\subseteq r_E^{-1}(F^0_i)$, we conclude that $F^1_i=r^{-1}_E(F_i^0)$, as desired.
\end{proof}

\begin{lemma}\label{sat}
Let $\{F_1,F_2\}$ be an admissible decomposition of a row-finite graph $E$.
Then $F_i^0\setminus(F_1^0\cap F^0_2)$ is saturated in $F_i$ and in $E$ for $i=1,2$.
\end{lemma}
\begin{proof}
If $F_i^0\setminus (F_1^0\cap F^0_2)$ were not saturated in $F_i$, then there would
exist a vertex $v$ in 
\begin{equation}
F_i^0\setminus (F_i^0\setminus (F_1^0\cap F^0_2))=F_1^0\cap F^0_2
\end{equation}
 such that
\begin{equation}
	s_{F_i}^{-1}(v)\neq\emptyset\quad\text{and}\quad r_{F_i}(s_{F_i}^{-1}(v))\subseteq F^0_i\setminus(F_1^0\cap F_2^0).
\end{equation}
Thus we would have a vertex in $F_1\cap F_2$ that is a sink in $F_1\cap F_2$ but not in $F_i$, which contradicts 
Definition~\ref{decompose}(\ref{sat2}).

Much in the same way, if $F_i^0\setminus (F_1^0\cap F^0_2)$ were not saturated in $E$,
then there would exist a vertex 
\begin{equation}
w\in E^0\setminus(F_i^0\setminus(F_1\cap F_2^0))=F^0_j\,,
\end{equation}
 where
$j\neq i$ and $j=1,2$, such that 
\begin{equation}
	s_E^{-1}(w)\neq\emptyset\quad\text{and}\quad r_E(s^{-1}_E(w))\subseteq F_i^0\setminus(F_1^0\cap F^0_2).
\end{equation}
Hence, there is $e\in s^{-1}_E(w)$ such that $r_E(e)\notin F_j^0$. As $E^1=F_i^1\cup F^1_j$, 
it follows that $e\in F^1_i$, so $w=s_E(e)\in F^0_i$.
Consequently, $w$ is a sink in $F_1\cap F_2$ but not in $F_i$, which again contradicts Definition~\ref{decompose}(\ref{sat2}).
\end{proof}

\section{Pullbacks of graph C*-algebras}

Let $\{F_1,F_2\}$ be an admissible decomposition of a row-finite graph $E$. Then, by Lemma~\ref{quot} and Lemma~\ref{sat}, 
we can take an advantage of the formula \eqref{quotient} to 
define the canonical quotient maps:
\begin{gather}
\pi_1:C^*(E)\longrightarrow C^*(E)/\overline{I_E(F_2^0\setminus F_1^0)}\cong C^*(F_1),\\
 \pi_2:C^*(E)\longrightarrow C^*(E)/\overline{I_E(F_1^0\setminus F_2^0)}\cong C^*(F_2),\\
\chi_1:C^*(F_1)\longrightarrow C^*(F_1)/\overline{I_{F_1}(F_1^0 \setminus (F_1^0\cap F_2^0))}\cong C^*(F_1\cap F_2),\\
\quad \chi_2:C^*(F_2)\longrightarrow C^*(F_2)/\overline{I_{F_2}(F_2^0 \setminus (F_1^0\cap F_2^0))}\cong C^*(F_1\cap F_2).
\end{gather}
Note that quotient maps are automatically $U(1)$-equivariant for the gauge action.
%\clearpage

This brings us to the main theorem:
\begin{theorem}\label{main}
Let $\{F_1,F_2\}$ be an admissible decomposition of a row-finite graph $E$. 
Then there exist canonical quotient gauge-equivariant $*$-homo\-mor\-phisms rendering the following diagram
\begin{equation}\label{pullback}
\begin{gathered}
\xymatrix{
& C^*(E) \ar[ld]_{\pi_1} \ar[rd]^{\pi_2} \\
C^*(F_1) \ar[rd]_{\chi_1}
& 
& C^*(F_2) \ar[ld]^{\chi_2}\\
& C^*(F_1\cap F_2)
}
\end{gathered}
\end{equation}
commutative. Moreover, this
is a \emph{pullback diagram} of $U(1)$-C*-algebras. 
\end{theorem}
\begin{proof}
Note first that all the canonical surjections in the diagram are well defined due to the admissibility conditions of the decomposition
of the graph
$E$ (see the discussion at the beginning of this section). The commutativity of the diagram is obvious as all maps are 
canonical surjections. Finally,
using \cite[Proposition~3.1]{g-p99} and the surjectivity of $\chi_1$ and~$\chi_2$, to prove that \eqref{pullback} is a pullback diagram,
 it suffices to show that
$\ker\pi_1\cap\ker\pi_2=\{0\}$ and that $\pi_2(\ker\pi_1)\subseteq\ker\chi_2$. 
 
Since $\ker\pi_1$ and $\ker\pi_2$ are closed ideals in a~C*-algebra, we know that
\begin{equation}\label{cap}
\ker\pi_1\cap\ker\pi_2=\ker\pi_1\ker\pi_2.
\end{equation}
Next, as $F_1^0\setminus F_2^0$ and $F_2^0\setminus F_1^0$ are saturated hereditary subsets of $E^0$, 
it follows from \eqref{quotient} that
\begin{equation}
\ker\pi_1=\overline{I_E(F_2^0\setminus F_1^0)}
\qquad\text{and}\qquad
\ker\pi_2=\overline{I_E(F_1^0\setminus F_2^0)}.
\end{equation}
Furthermore, using the characterization \eqref{herideal} of ideals generated by hereditary subsets, 
we know  that an arbitrary element of $\ker\pi_1\ker\pi_2$ is in the closed linear span of elements  of the form 
$S_\alpha S_\beta^*S_\gamma S_\delta^*$, where $\alpha, \beta\in \text{Path}(E)$ with 
\begin{equation}
r_{P\!E}(\alpha)=r_{P\!E}(\beta)\in F_2^0\setminus F_1^0\,,
\end{equation}
and $\gamma,\delta\in \text{Path}(E)$ with 
\begin{equation}
r_{P\!E}(\gamma)=r_{P\!E}(\delta)\in F_1^0\setminus F_2^0\,.
\end{equation}
The conclusion $\ker\pi_1\cap\ker\pi_2=\{0\}$ follows from the analysis of all
possible paths satisfying the above conditions.

Indeed, it follows from \eqref{pathalg} that $S^*_\beta S_\gamma\neq 0$ is possible only if~$s_{P\!E}(\beta)=s_{P\!E}(\gamma)$. 
As $E^1=F^1_1\cup F^1_2$, $r_{P\!E}(\beta)\in F_2^0\setminus F_1^0$ and $r_{P\!E}(\gamma)\in F_1^0\setminus F_2^0$, 
if $\beta=e_1\ldots e_m$ 
and $\gamma=f_1\ldots f_n$, we infer that 
\begin{equation}
r_E(e_{m-1})=s_E(e_m)\in F_2^0\quad \text{and}\quad r_E(f_{n-1})=s_E(f_n)\in F_1^0\,.
\end{equation}
 Hence $r_E(e_{m-1})\in
F^0_1\cap F^0_2$ or $r_E(e_{m-1})\in F^0_2\setminus F^0_1$. Now, we continue by induction using Definition~\ref{decompose}(3)
for the intersection case of the alternative. This brings us to conclusion that $s_{P\!E}(\beta)\in  F_2^0$. Much in the same way, we argue that
$s_{P\!E}(\gamma)\in  F_1^0$. It follows that
$s_{P\!E}(\beta)=s_{P\!E}(\gamma)\in F^0_1\cap F^0_2$. Furthermore, as $r_{P\!E}(\beta)\in F^0_2\setminus F^0_1$ and
$r_{P\!E}(\gamma)\in F^0_1\setminus F^0_2$, we conclude  that~$\beta\neq\gamma$, so
 there exists the smallest index $i$ such that $e_i\neq f_i$. Now, remembering the relations
\eqref{eq:CK1} and \eqref{pathalg},
we compute
\begin{align}
S^*_\beta S_\gamma
&= S^*_{e_{i+1}\ldots e_m}S_{e_i}^*S_{e_{i-1}}^*\ldots S_{e_1}^* S_{e_1}\ldots S_{e_{i-1}}S_{f_i}S_{f_{i+1}\ldots f_n}
\nonumber\\
&=S^*_{e_{i+1}\ldots e_m}S_{e_i}^*S_{f_i}S_{f_{i+1}\ldots f_n}
\nonumber\\
&=0.
\end{align}
\noindent
Finally, if $\beta$ or $\gamma$ is a path of length zero, i.e.\ a vertex, then it is straightforward to conclude that %$s_E(\beta)\neq s_E(\gamma)$.
$S^*_\beta S_\gamma=0$.

Next, taking again an advantage of  \eqref{herideal} and \eqref{quotient}, we obtain
\begin{align*}
%\ker\pi_1&=\overline{I_E(F_2^0\setminus F_1^0)}
%=\overline{{\rm span}}
%\{S_\alpha S_\beta^*\;|\; \alpha,\beta\in{\rm Path}(E), 
%r(\alpha)=r(\beta)\in F_2^0\setminus F_1^0\},\\
\ker\chi_2&=\overline{I_{F_2}(F_2^0\setminus F_1^0)}
=\overline{{\rm span}}\{S_\alpha S_\beta^*\;|\; \alpha,\beta\in{\rm Path}(F_2), 
r_{P\!F_2}(\alpha)=r_{P\!F_2}(\beta)\in F_2^0\setminus F_1^0\}.
\end{align*}
Any element of $I_{F_2}(F_2^0\setminus F_1^0)$ is an element of $I_{E}(F_2^0\setminus F_1^0)$,
and $\pi_2(S_\alpha)=S_\alpha$ for all $\alpha\in{\rm Path}(F_2)$. Hence
$\pi_2(I_{E}(F_2^0\setminus F_1^0))\subseteq I_{F_2}(F_2^0\setminus F_1^0)$.
Finally, from the continuity of $\pi_2$, we conclude that 
$\pi_2(\ker\pi_1)\subseteq\ker\chi_2$. 
\end{proof}

\begin{remark}\rm
One can also prove Theorem~\ref{main} in the setting of Leavitt path algebras~\cite{aasm17}. 
A proof of the Leavitt version of Theorem~\ref{main} is completely analogous due to \cite[Corollary~2.5.11]{aasm17}.
\end{remark}

\section{Examples}

We end the paper by providing motivating examples  from noncommutative  topology.
\subsection{Even quantum spheres}
Not only the graph at the top of the diagram \eqref{podles} representing
 the generic Podle\'s quantum sphere \cite{p-p87} admits a natural admissible decomposition, but
also the finite graphs $L_{2n}$ \cite[Section~5.1]{hs02} representing, respectively, 
the C*-algebras $C(S^{2n}_q)$ of all even quantum spheres enjoy natural admissible decompositions~$\{F^1_{2n},F^2_{2n}\}$.
Here $C^*(F^1_{2n})=C^*(F^2_{2n})$ coincides with
the C*-algebra $C(B^{2n}_q)$ of the Hong-Szyma\'nski quantum $2n$-ball \cite[Section~3.1]{hs08}, 
and $C^*(F^1_{2n}\cap F^2_{2n})$
coincides \cite[Appendix~A]{hs02} with the C*-algebra $C(S^{2n-1}_q)$ of the boundary Vaksman-Soibelman quantum odd sphere \cite{vs91}. 
Thus we recover in terms of graphs the classical fact that an even sphere is a gluing 
of even balls over the boundary odd sphere.
 
As  Theorem~\ref{main} applies, we infer that 
the diagram
\begin{equation}%\label{}
\begin{gathered}
\xymatrix{
& C(S^{2n}_q) \ar[ld]_{\pi_1} \ar[rd]^{\pi_2} \\
C(B^{2n}_q) \ar[rd]_{\chi_1}
& 
& C(B^{2n}_q) \ar[ld]^{\chi_2}\\
& C(S^{2n-1}_q)~
}
\end{gathered}
\end{equation}
is a pullback diagram. This fact was already proved in \cite[Proposition~5.1]{hs08} by direct considerations of generators and relations.

The case $n=3$ is illustrated by the diagram:
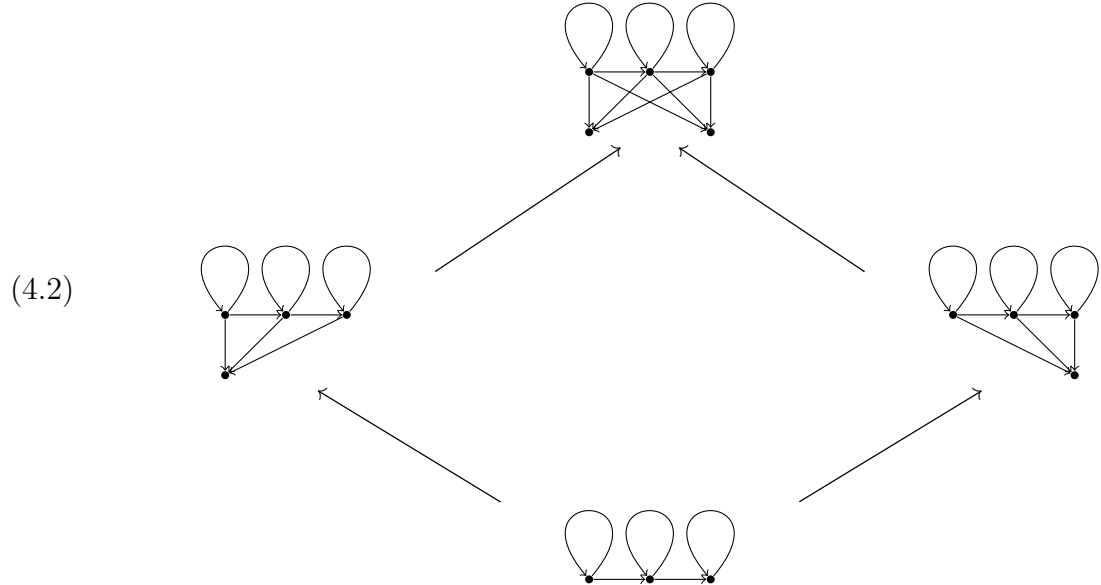
\begin{equation}
\begin{gathered}
\xymatrix{
&
\begin{tikzpicture}[scale=0.4,auto,swap]
\centering
\tikzstyle{vertex}=[circle,fill=black,minimum size=3pt,inner sep=0pt]
\tikzstyle{edge}=[draw,->]
\tikzset{every loop/.style={min distance=20mm,in=130,out=50,looseness=50}}
    \node[vertex] (1) at (-2,0) {};
    \node[vertex] (2) at (0,0) {};
    \node[vertex] (3) at (2,0) {};
    \node[vertex] (4) at (-2,-2) {};
    \node[vertex] (5) at (2,-2) {};
    \path (1) edge [edge, anchor=center, loop above] node {} (1);
    \path (2) edge [edge, anchor=center, loop above] node {} (2);
    \path (3) edge [edge, anchor=center, loop above] node {} (3);
    \path (1) edge [edge] node {} (2);
    \path (1) edge [edge] node {} (4);
    \path (1) edge [edge] node {} (5);
    \path (2) edge [edge] node {} (3);
    \path (2) edge [edge] node {} (4);
    \path (2) edge [edge] node {} (5);
    \path (3) edge [edge] node {} (4);
    \path (3) edge [edge] node {} (5);
\end{tikzpicture}
&\\
\begin{tikzpicture}[scale=0.4,auto,swap]
\centering
\tikzstyle{vertex}=[circle,fill=black,minimum size=3pt,inner sep=0pt]
\tikzstyle{edge}=[draw,->]
\tikzset{every loop/.style={min distance=20mm,in=130,out=50,looseness=50}}
    \node[vertex] (1) at (-2,0) {};
    \node[vertex] (2) at (0,0) {};
    \node[vertex] (3) at (2,0) {};
    \node[vertex] (4) at (-2,-2) {};
    \path (1) edge [edge, anchor=center, loop above] node {} (1);
    \path (2) edge [edge, anchor=center, loop above] node {} (2);
    \path (3) edge [edge, anchor=center, loop above] node {} (3);
    \path (1) edge [edge] node {} (2);
    \path (1) edge [edge] node {} (4);
    \path (2) edge [edge] node {} (3);
    \path (2) edge [edge] node {} (4);
    \path (3) edge [edge] node {} (4);
\end{tikzpicture}
\ar[ur]& & 
\begin{tikzpicture}[scale=0.4,auto,swap]
\centering
\tikzstyle{vertex}=[circle,fill=black,minimum size=3pt,inner sep=0pt]
\tikzstyle{edge}=[draw,->]
\tikzset{every loop/.style={min distance=20mm,in=130,out=50,looseness=50}}
    \node[vertex] (1) at (-2,0) {};
    \node[vertex] (2) at (0,0) {};
    \node[vertex] (3) at (2,0) {};
    \node[vertex] (5) at (2,-2) {};
    \path (1) edge [edge, anchor=center, loop above] node {} (1);
    \path (2) edge [edge, anchor=center, loop above] node {} (2);
    \path (3) edge [edge, anchor=center, loop above] node {} (3);
    \path (1) edge [edge] node {} (2);
    \path (1) edge [edge] node {} (5);
    \path (2) edge [edge] node {} (3);
    \path (2) edge [edge] node {} (5);
    \path (3) edge [edge] node {} (5);
\end{tikzpicture}
\ar[ul]\\
&
\begin{tikzpicture}[scale=0.4,auto,swap]
\centering
\tikzstyle{vertex}=[circle,fill=black,minimum size=3pt,inner sep=0pt]
\tikzstyle{edge}=[draw,->]
\tikzset{every loop/.style={min distance=20mm,in=130,out=50,looseness=50}}
    \node[vertex] (1) at (-2,0) {};
    \node[vertex] (2) at (0,0) {};
    \node[vertex] (3) at (2,0) {};
    \path (1) edge [edge, anchor=center, loop above] node {} (1);
    \path (2) edge [edge, anchor=center, loop above] node {} (2);
    \path (3) edge [edge, anchor=center, loop above] node {} (3);
    \path (1) edge [edge] node {} (2);
    \path (2) edge [edge] node {} (3);
\end{tikzpicture}
\ar[ur]\ar[ul]&
.}
\end{gathered}
\end{equation}
\subsection{Quantum lens space $L^3_q(l;1,l)$}
The C*-algebra $C(L^3_q(l;1,l))$ of the quantum lens space $L^3_q(l;1,l)$ can be viewed as the graph C*-algebra (e.g., see \cite{bs16})
of the graph $L^3_l$:
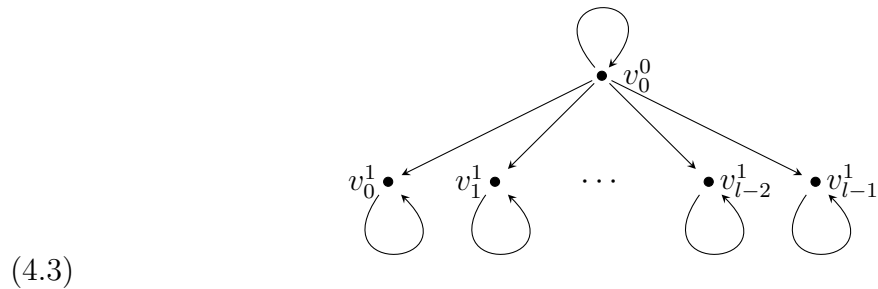
\begin{equation}
\begin{tikzpicture}[>=stealth,node distance=40pt,main node/.style={circle,inner sep=2pt},
freccia/.style={->,shorten >=1pt,shorten <=1pt},
ciclo/.style={out=130, in=50, loop, distance=40pt, ->},
ciclo2/.style={out=235, in=315, loop, distance=40pt, ->}]

      \node[main node] (1) {};
      \node (2) [below of=1] {$\cdots$};
      \node (3) [left of=2]{};
      \node (4) [left of=3]{};
      \node (5) [right of=2]{};
     \node (6) [right of=5] {};

      \filldraw (1) circle (0.06) node[right] {$\ v_0^0$};
      \filldraw (4) circle (0.06) node[left] {$v_0^1$};
      \filldraw (3) circle (0.06) node[left] {$v_1^1$};
      \filldraw (5) circle (0.06) node[right] {$v_{l-2}^1$};
      \filldraw (6) circle (0.06) node[right] {$v_{l-1}^1$};

      \path[freccia] (1) edge[ciclo] (1);
			\path[freccia] (4) edge[ciclo2] (4);
			\path[freccia] (3) edge[ciclo2] (3);
			\path[freccia] (5) edge[ciclo2] (5);
			\path[freccia] (6) edge[ciclo2] (6);

      \path[freccia] (1) edge (4);         							
\path[freccia] (1) edge (3)	; 
\path[freccia] (1) edge (5)	;           
\path[freccia] (1) edge (6)	;           
\end{tikzpicture}.
\end{equation}
\noindent
The graph $L^3_l$ enjoys an admissible decomposition $\{L^3_k,L^3_{l-k}\}$, where 
$k\in\{1,\ldots l-1\}$,
yielding, by Theorem~\ref{main}, the pullback diagram: 
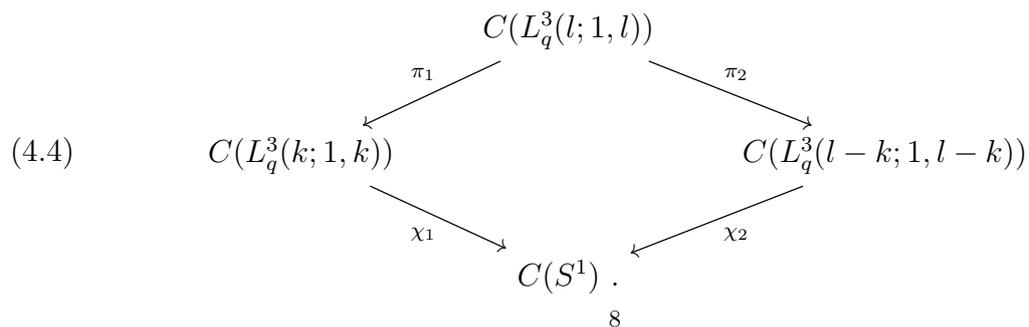
\begin{equation}
\begin{gathered}
\xymatrix{
& C(L^{3}_q(l;1,l)) \ar[ld]_{\pi_1} \ar[rd]^{\pi_2} \\
C(L^{3}_q(k;1,k)) \ar[rd]_{\chi_1}
& 
& C(L^{3}_q(l-k;1,l-k)) \ar[ld]^{\chi_2}\\
& C(S^{1})~.
}
\end{gathered}
\end{equation}

Recall that $C^*(L^3_1)\cong C(S^3_q)$, so for $l=2$ we obtain the following pullback diagram:
\begin{equation}
\begin{gathered}
\xymatrix{
& C(L^{3}_q(2;1,2)) \ar[ld]_{\pi_1} \ar[rd]^{\pi_2} \\
C(S^3_q) \ar[rd]_{\chi_1}
& 
& C(S^3_q) \ar[ld]^{\chi_2}\\
& C(S^{1})~.
}
\end{gathered}
\end{equation}
Since the above diagram is $U(1)$-equivariant, it induces a pullback diagram for\linebreak
 \mbox{$U(1)$-fixed-point} subalgebras:
\begin{equation}
\begin{gathered}
\xymatrix{
& C(\mathbb{W}P^1_q(1,2)) \ar[ld]_{\pi_1} \ar[rd]^{\pi_2} \\
C(\mathbb{C}P^1_q) \ar[rd]_{\chi_1}
& 
& C(\mathbb{C}P^1_q) \ar[ld]^{\chi_2}\\
& \mathbb{C}~.
}
\end{gathered}
\end{equation}
Here $C(\mathbb{C}P^1_q)$ and $C(\mathbb{W}P^1_q(1,2))$ denote the quantum
complex projective space (see \cite[Section~2.3]{hs02})  and the quantum weighted projective space (see \cite[Section~3]{bs16}), respectively. 
Interestingly, the C*-algebras in the above diagram can be viewed as graph C*-algebras, and
an infinite graph 
representing $C(\mathbb{W}P^1_q(1,2))$ is a pushout of infinite graphs representing $C(\mathbb{C}P^1_q)$ over the graph 
consisting of one vertex and no edges representing~$\mathbb{C}$
(see Diagram~\eqref{inf} below). 
However, this example is beyond the scope of Theorem~\ref{main}, because the above diagram is no longer 
$U(1)$-equivariant and the infinite graphs are not row finite.

\begin{equation}\label{inf}
\begin{gathered}
\xymatrix{
&
\begin{tikzpicture}[scale=0.4,auto,swap]
\centering
\tikzstyle{vertex}=[circle,fill=black,minimum size=3pt,inner sep=0pt]
\tikzstyle{nothing}=[fill=white,minimum size=3pt,inner sep=0pt]
\tikzstyle{edge}=[draw,->]
\tikzstyle{unedge}=[draw,-]
    \node[vertex] (1) at (0,0) {};
    \node[nothing] (2) at (-2,0) {($\infty$)};
    \node[nothing] (3) at (2,0) {($\infty$)};
    \node[vertex] (4) at (-4,0) {};
    \node[vertex] (5) at (4,0) {};
    \path (1) edge [unedge] node {} (2);
    \path (2) edge [edge] node {} (4);
    \path (1) edge [unedge] node {} (3);
    \path (3) edge [edge] node {} (5);
\end{tikzpicture}
&\\
\begin{tikzpicture}[scale=0.4,auto,swap]
\centering
\tikzstyle{vertex}=[circle,fill=black,minimum size=3pt,inner sep=0pt]
\tikzstyle{nothing}=[fill=white,minimum size=3pt,inner sep=0pt]
\tikzstyle{edge}=[draw,<-]
\tikzstyle{unedge}=[draw,-]
    \node[vertex] (1) at (0,0) {};
    \node[nothing] (2) at (2,0) {($\infty$)};
    \node[vertex] (3) at (4,0) {};
    \path (1) edge [edge] node {} (2);
    \path (2) edge [unedge] node {} (3);
\end{tikzpicture}
\ar[ur]& & 
\begin{tikzpicture}[scale=0.4,auto,swap]
\centering
\tikzstyle{vertex}=[circle,fill=black,minimum size=3pt,inner sep=0pt]
\tikzstyle{nothing}=[fill=white,minimum size=3pt,inner sep=0pt]
\tikzstyle{edge}=[draw,<-]
\tikzstyle{unedge}=[draw,-]
    \node[vertex] (1) at (0,0) {};
    \node[nothing] (2) at (2,0) {($\infty$)};
    \node[vertex] (3) at (4,0) {};
    \path (3) edge [edge] node {} (2);
    \path (2) edge [unedge] node {} (1);
\end{tikzpicture}
\ar[ul]\\
&
\begin{tikzpicture}[scale=0.4,auto,swap]
\centering
\tikzstyle{vertex}=[circle,fill=black,minimum size=3pt,inner sep=0pt]
    \node[vertex] (1) at (0,0) {};
\end{tikzpicture}
\ar[ur]\ar[ul]&}
\end{gathered}
\end{equation}
Here edges with  ($\infty$) denote countably infinitely many edges.
\clearpage

\section*{Acknowledgements}\noindent
The work on this project was  partially supported by NCN-grant 2015/19/B/ST1/03098
(Piotr M.\ Hajac, Mariusz Tobolski)
and by a Simons Foundation Collaboration Grant (Sarah Reznikoff). 
Piotr M.\ Hajac is very grateful to Kansas State University for its hospitality and financial 
support provided by this Simons Foundation Collaboration Grant.  
It is a pleasure to thank  Carla Farsi for drawing our attention to using pushouts of graphs
in the theory of free groups, Tatiana Gateva--Ivanova for a helpful discussion concerning admissible decompositions of graphs, 
and  Aidan Sims for making us aware on how to prove our main result in a different context.

\end{document}